\documentclass{alggeom}

\usepackage{mathrsfs}
\usepackage{amscd}
\usepackage{amsmath}
\usepackage{amssymb}
\usepackage{latexsym}
\usepackage{comment}
\usepackage{amscd}
\usepackage{wasysym}  
\usepackage{tikz}
\usetikzlibrary{matrix,arrows}
\usepackage{tikz-cd}
\usepackage{appendix}
\usepackage[nice]{nicefrac}
\usepackage{fix-cm} 
\usepackage{comment}
\usepackage{amscd}
\usepackage{wasysym}  
\usepackage{tikz}
\usetikzlibrary{matrix,arrows}
\usepackage{tikz-cd}
\usepackage{appendix}
\usepackage[nice]{nicefrac}

\usepackage{marvosym}
\usepackage{tikz}
\usepackage[all,cmtip]{xy}
\usetikzlibrary{matrix,arrows}

\usepackage[pagebackref,hyperindex,linktocpage=true]{hyperref}
\hypersetup{
    colorlinks,
    linkcolor={blue!20!black},
    citecolor={blue!20!black},
    urlcolor={blue!20!black}
}

\usepackage{color}

\usepackage{MnSymbol}
\usetikzlibrary{decorations.markings}

\makeatletter
\tikzcdset{
  open/.code     = {\tikzcdset{hook, circled};},
  closed/.code   = {\tikzcdset{hook, slashed};},
  open'/.code    = {\tikzcdset{hook', circled};},
  closed'/.code  = {\tikzcdset{hook', slashed};},
  circled/.code  = {\tikzcdset{markwith = {\draw (0,0) circle (.375ex);}};},
  slashed/.code  = {\tikzcdset{markwith = {\draw[-] (-.4ex,-.4ex) -- (.4ex,.4ex);}};},
  markwith/.code ={
    \pgfutil@ifundefined%
    {tikz@library@decorations.markings@loaded}%
    {\pgfutil@packageerror{tikz-cd}{You need to say %
      \string\usetikzlibrary{decorations.markings} to use arrows with markings}{}}{}%
    \pgfkeysalso{/tikz/postaction = {
      /tikz/decorate,
      /tikz/decoration={markings, mark = at position 0.5 with {#1}}}
    }
  },
}
\makeatother

\renewcommand{\geq}{\geqslant}
\renewcommand{\leq}{\leqslant}

\newtheorem{theo}{Theorem}[section]
\newtheorem{defi}[theo]{Definition}
\newtheorem{exam}[theo]{Example}
\newtheorem{prop}[theo]{Proposition}

\newtheorem{lemm}[theo]{Lemma}

\newtheorem{coro}[theo]{Corollary}
\newtheorem{rema}[theo]{Remark}
\newtheorem{conj}[theo]{Conjecture}

\newcommand{\N}{\mathbb{N}}

\newcommand{\I}{\mathbb{I}}
\newcommand{\pp}{\mathbb{P}}

\title[Conjecture about prime numbers]{Conjecture: the set of prime numbers is supernatural}

\begin{document}

\author{Arnaud Mayeux}
\email{arnaud.mayeux@mail.huji.ac.il}
\address{Einstein Institute of Mathematics,
The Hebrew University of Jerusalem,
 Jerusalem, 9190401, Israel}


\begin{abstract}
 Prime numbers are fascinating by the way they appear in the set of natural numbers. Despite several results enlighting us about their repartition, the set of prime numbers is often informally qualified as misterious. In the present paper, we introduce a formalism allowing to state a formal conjecture: the set of prime numbers is supernatural. Our conjecture has no analog in the existing literature. This conjecture is expected to be a hard challenge for any kind of intelligence. We also find a Fermat-like function giving more prime numbers than Fermat's function.  
\end{abstract}

\maketitle

\tableofcontents

\section*{Introduction}
 \begin{sloppypar}
Let $\pp$ be the set of prime numbers.
Fermat  \cite{fermat} conjectured that the set  $\{ 2^{2^n}+1 \mid n \text{ positive integers} \} $ consists entirely of prime numbers. Euler \cite{euler1, euler2} disproved this conjecture showing that $2^{2^5}+1$ is composite. Let $\mathbb{I}$ be the set of positive non-zero integers. A function $f : \I \to \I $ with variable $n$ is a natural function if it is constructed from formal (and coherent) combinations of the symbols $n$, constants $ a \in \I$, and the operations $+ , \times , \wedge$. Here $\wedge$ means elevation to the power. We prove that every natural function is constant or strictly increasing. We then conjecture that for every non-constant natural function $f$, we have $f (\I ) \not \subset \pp$.  This is our conjecture. Evidences for this conjecture are intuition about the irregularity of the set of prime numbers, and that it is verified for all natural functions that we have tested. \end{sloppypar}

 In this text we formally introduce notation and definitions. Then we state our conjecture. Our conjecture was formulated in 2013. We explain that this conjecture can be reformulated as in the title of the article: $\pp$ is supernatural. Proposition \ref{infini} shows that if the conjecture is true then $f(\I) \cap (\I \setminus \pp )$ is an infinite set for all non-constant natural function $f$. In particular, this "canonical" conjecture implies the existence of infinite composite Fermat numbers (an open question today), cf. Corollary \ref{Fermatcor}. Proposition \ref{incredi} explains that if the conjecture is wrong, this would provide a natural function so incredible that we can not believe it. We finish the paper by experiments on the conjecture. One of our conclusions is that Fermat numbers do not have something very special, for example numbers of the forms $2^{2^n}+93$ provide prime numbers for $1 \leq n <7$.

\section{Notation and definition}
\begin{sloppypar}
Let $\N = \{0,1,2,3,4,5,6,7, \ldots \} $ be the set of positive integers.
Let $\I=\{1,2,3,4,5,6,7, \ldots\} $ be the set of non-zero positive integers.  Let $ \pp =\{2,3,5,7,11,13,17,\ldots\} $ be the set of prime numbers. Let $\mathscr{F}$ be the set of functions from $\I$ to $\I$. \end{sloppypar}

\begin{defi}An elevation structure is a $4$-uple $(E,+,\times , \wedge )$ where $E$ is a set and \begin{enumerate}
\item $+$ is a map $E \times E \to E$ sending $a,b$ to $a+b$,
\item $\times $ is a map $E \times E \to E$ sending $a,b$ to $a \times b$,
\item $\wedge $ is a map $E \times E \to E $ sending $a,b $ to $a \wedge b $,
\end{enumerate}
such that $\forall ~ a,b,c \in E$ the following hold 
(i) $a+b=b+a$,
(ii) $a+(b+c)=(a+b)+c$,
(iii) $b \times a = a \times b $,
(iv) $a\times (b \times c ) = (a \times b ) \times c $,
(v) $ a \times (b+c) = a \times b + a \times c $,
(vi) $(a \wedge b ) \times (a \wedge c) = a \wedge ({b+c})$,
(vii) $(a \wedge b ) \wedge c = a \wedge ( b \times c )$.
 We write $a_1 \wedge a_2 \wedge \ldots  \wedge a_n $ instead of $a_1 \wedge (a_2 \wedge (\ldots \wedge a_n))$, be careful that $a \wedge ( b \wedge c ) \ne ( a \wedge b ) \wedge c $ in general. We  write often $a_1^{a_2^{\udots ^{a_n}}}$ instead of $a_1 \wedge a_2 \wedge \ldots  \wedge a_n $. For example $a^{b^c}$ means $a \wedge (b \wedge c )$.

 A morphism of elevation structures from $(E,+, \times , \wedge ) $ to $(F ,+, \times , \wedge )$ is a map $M:E \to F$ such that $\forall ~a,b \in E$, we have $M (a+b) = M (a) + M (b)$, $~$ $M(a \times b) = M (a) \times M (b)$ and $M (a \wedge b) = M (a) \wedge M (b)$. We obtain a category.
\end{defi}

\begin{exam} Note that $\I,+,\times , \wedge$ is an elevation structure. Similarly, $\mathscr{F},+, \times , \wedge$ is an elevation structure with $(f \wedge g )(n) := f(n) \wedge g(n)$.
 Let $a \in \I$, then the evaluation map $E_a:\mathscr{F} \to \I$, $f \mapsto f(a)$, is a morphism of elevation structures from  $\mathscr{F} ,+, \times , \wedge $ to $ \I, +,\times , \wedge$.
\end{exam}

Let $P (\mathscr{F}) $ be the set of all parts of  $\mathscr{F}$. Let $\mathscr{F}_s \subset \mathscr{F}$ be the part of functions constitued of the identity map $n \mapsto n$ and the constant maps $n \mapsto a$ for every $a \in \I$, we also call it the part of symbols.
We now define some operators on $P(\mathscr{F} )$.

\begin{defi} $~~$

\begin{enumerate}
\item Let $A_+$ be the map sending $\mathscr{P} \in P(\mathscr{F} ) $ to \[ \mathscr{P} \cup \{ f \in \mathscr{F} \mid \exists g\in P, \exists h \in P  \text{ and } f=g+h \}  .\]
\item Let $A_{\times}$ be the map sending $\mathscr{P}  \in P(\mathscr{F} ) $ to \[\mathscr{P}  \cup  \{ f \in \mathscr{F} \mid \exists g\in P, \exists h \in P  \text{ and } f=g \times h \} .\]
\item Let $A_{\wedge}$ be the map sending $\mathscr{P}  \in P(\mathscr{F} ) $ to \[ \mathscr{P}  \cup \{ f \in \mathscr{F} \mid \exists g\in P, \exists h \in P  \text{ and } f=g \wedge h \} .\]
\end{enumerate}

\end{defi}

We are now able to define the set of natural functions $\mathscr{F} _{Natural}$. Let $\Sigma$ be the set of all finite words with letters $A_+ , A_{\times} , A _{\wedge}$.

\begin{defi} \label{natfct} We put $\mathscr{F} _{Natural} = \displaystyle \bigcup _{\sigma \in \Sigma } \sigma ( \mathscr{F}_s)$, this is an element in $P(\mathscr{F} ) $. The $4$-uple $\mathscr{F} _{Natural} , + , \times , \wedge $ is an elevation structure.

\end{defi}

\begin{exam}The inclusion $\mathscr{F}_s \subset \mathscr{F} _{Natural}$ holds.
 The set of polynomial functions on $\I$ $\mathscr{F} _{Polynomial}$ is contained in $\mathscr{F} _{Natural}$. More precisely it is contained in $ \displaystyle \bigcup _{j,k \geq 0} A_+^j A_{\times} ^k (\mathscr{F} _s )$. 
 Fermat's function $n \mapsto 2^{2^n} +1 $ is contained in $\mathscr{F} _{Natural}$. More precisely it is contained in $A_+ A_{\wedge} ^2 ( \mathscr{F} _s )$.
 The functions $
n \mapsto n^n+n+1 ,
n \mapsto 7^n+6 ,
n  \mapsto  n^{4n^n +n^{23}+2^n+8} +n+3,
n  \mapsto 2^{2^{2^{2^{2^{2^n}}}}}+1 $
are natural functions.

\end{exam}

\begin{prop} \label{increa} A natural function $f \in \mathscr{F} _{Natural}$ is constant or strictly increasing.

\end{prop}

\begin{proof} By definition a function $f $ is in $\mathscr{F} _{Natural}$ is and only if there exists a finite word $\sigma$ in letters $A_+ , A_{\times} , A_{\wedge}$ such that $f \in \sigma (\mathscr{F} _s )$.  We need the following definition.

\begin{defi} The length of a natural function $f \in \mathscr{F}_{Natural}$ is the number \[ \min \{n \in \N \mid  f \in \sigma (\mathscr{F}_s) \text{ with $\sigma$ a word of length $n$} \}, \] it is well defined.

\end{defi} Let us prove our result by induction on length. Consider the following assertion.
\begin{center}

 \textit{$P_n :$ Proposition \ref{increa} is true for every $f \in \mathscr{F}_{Natural}$ of length $\leq n$.}
\end{center}
 
 Let us prove by induction that $P_n$ is true for every $n \in \N$. We need the following Lemmas.
 
  \begin{lemm} \label{lem0} Let $h,g \in \mathscr{F}_{Natural}$ be strictly increasing, then $h+g , h\times g , h \wedge g  $ are strictly increasing.
\end{lemm} 
\begin{proof}
Let $i>j$ be integers in $\I$, then $h(i)>h(j)$ and $g(i)> g (j)$. So $h(i)+g(i) > h(j)+g(j)$, consequently $h+g$ is strictly increasing. Since $ g(i) > g(j) \geq 1 $ and $h(i) > h(j) \geq 1$, we have $h(i)\times g(i) > h(j)\times g(j)$ and consequently $h\times g $ is strictly increasing. We also have $h(i)^{g(i)} > h(j) ^{g(j)}$ and consequently $h \wedge g$ is strictly increasing.
\end{proof} 
 
  \begin{lemm} \label{lem1} Let $h \in \mathscr{F}_{Natural}$ be strictly increasing and $g  \in \mathscr{F} _{Natural}$ be constant, then $h+g , h\times g , h \wedge g  $ are strictly increasing.
\end{lemm} 
\begin{proof}
Let $i>j$ be integers in $\I$, then $h(i)>h(j)$ and $g(i)= g (j)$. So $h(i)+g(i) > h(j)+g(j)$, consequently $h+g$ is strictly increasing. Since $g(i) = g(j) \geq 1$ and $h(i) > h(j)$, we have $h(i)\times g(i) > h(j)\times g(j)$ and consequently $h\times g $ is strictly increasing. We have $h(i)^{g(i)} > h(j) ^{g(j)}$ and consequently $h \wedge g$ is strictly increasing.
\end{proof} 

\begin{lemm} \label{lem2} Let $h \in \mathscr{F}_{Natural}$ be strictly increasing and $g  \in \mathscr{F} _{Natural}$ be constant, then $g \wedge h $ is strictly increasing except if $g=1$. If $g=1$, $g \wedge h $ is constant.
\end{lemm}

\begin{proof} Let $i>j$ be integers in $\I$. If $g=1$ then $g(i) ^{ h(i)}=g(j) ^{ h(j)}=1$ and $g \wedge h$ is constant. In the other case $g(i)=g(j) >1$ and $h(i)>h(j)$, consequently $g(i) ^ {h(i)} > g (j) ^{ h(j)}$. This finishes the proof.
\end{proof}
 
 Now let us start our induction. If $n=0$, then $f$ is constant or the identity map and thus satisfies $P_0$. Now assume that $P_n$ is true and let us prove that $P_{n+1}$ is true. Let $f $ of length $n+1$.  By definition, there are $h,g $ of length $\leq n$ such that $f=h+g$ or $f=h \times g $ or $f = g \wedge h$ or $f= h \wedge g$. By induction hypothesis $h$ and $g$ are constant or strictly increasing.  Assume first both are constant. Then $f$ is obviously constant. If both are strictly increasing, then $f$ is strictly increasing by Lemma \ref{lem0}. If $h$ is strictly increasing and $g$ is constant, then by Lemma \ref{lem1}, $h+g$ , $h \times g$ and $h \wedge g$ are strictly increasing, and by Lemma \ref{lem2}, $g \wedge h$ is strictly increasing or constant. If $h$ is constant and $g$ is strictly increasing, we deduce, in the same way, that $h+g$ , $h \times g$, $h \wedge g$ and $g \wedge h$ are strictly increasing or constant. Consequently, $f$ is constant or strictly increasing. This finishes the induction and the proof of Proposition.

\end{proof}

\begin{rema} The map $\N \to \N$, $n \mapsto n^n$ is neither constant nor strictly increasing since $0^0=1^1\neq 2^2$. Its restriction to $\I$ is a strictly increasing natural function.

\end{rema}

\section{Statement of the conjecture}

\subsection{Statement}
We now state our conjecture.

\begin{conj} \label{conj}
 Let $f \in \mathscr{F}_{Natural} $ be a non-constant natural function. 
Then \[f( \I) \not \subset \pp.\]
\end{conj}

\begin{prop} \label{evidence} Conjecture  \ref{conj} is true in the following cases.

\begin{enumerate}
\item Let $f \in \mathscr{F} _{Polynomial}$, assume $f$ is non-constant, then $f (\I) \not \subset \pp$.
\item Let $a \in \I _{>1}$, $b \in \N$ and $f  $ be $n \mapsto a^n +b$, then $f (\I ) \not \subset \pp$.
\item Let $f$ be $n \mapsto 2^{2^n}+1$, the Fermat function. Then $f (\I) \not \subset \pp$.
\item Let $f$ be a natural function considered in the experiment (§\ref{secexp}). Then $f (\I) \not \subset \pp$.
\end{enumerate}
\end{prop}

\begin{proof} \begin{enumerate} \item Since $f$ is non-constant, it is strictly increasing by \ref{increa}, so there exits $n \in \I$ such that $f(n) >1$. Let $p \in \pp$ such that $p \mid f(n)$. Then $f(n) =0 \mod p$. Since $f$ is polynomial, we have $f(n+p) = f(n) \mod p.$ So we have $ p \mid f(n) , p \mid f(n+p) ,1< f(n) < f(n+p) .$ This implies that $f(n+p)$ is not a prime number. So $f(\I) \not \subset \pp$.
\item Let $n\in \I _{>1}$. Then $f(n)=a^n +b >1$.  Let $p\in \pp$ such that $p \mid f(n)$. Then $a^n+b =0 ~ \mod p. $ Assume first that $p \nmid a$, then by Fermat's little theorem $a^{n+ (p-1)} +b = a^n +b =0 \mod p.$ So we have $p \mid f(n),  p \mid f (n+(p-1)) ,  1 < f(n) < f(n+(p-1)) .$ This implies that $f (n+p-1) $ is not a prime number. Now if $p \mid a$, then $p < f(n)$ and $f(n)$ is not a prime number. So we have proved that $f(\I) \not \subset \pp$.
\item  We have $f(5)=2^{2^5}+1 =4294967297=641 \times 6700417$, as Euler \cite{euler1, euler2} computed. So $f(\I) \not \subset \pp$.
\item See the experiment.
\end{enumerate}
\end{proof}

\subsection{Reformulation: $\pp$ is supernatural} 

An infinite subset of $\I$ is said to be natural if it is of the form $f(\I)$ where $f\in \mathscr{F}_{Natural}$ is non-constant. An infinite subset of $\I$ is said to be supernatural if it does not contain any natural subset of $\I$. Our conjecture is now equivalent to the following statement: $\pp$ is supernatural.

\section{Implications and remarks}

In this section we discuss implications of our conjecture and comment it.

\begin{prop} \label{infini} Assume Conjecture \ref{conj} is true. Let $f \in \mathscr{F} _{Natural}$ non-constant, then $\{ x \in f (\I)  \mid x \text{ is not a prime number} \}$ is infinite.

\end{prop}
\begin{proof}By \ref{increa}, $f$ is strictly increasing. It is enough to prove the following statement. \textit{For all $k \in \I$, there exists $q\in \N _{>k}$, such that $f(q)$ is not a prime number. }
So let $k \in \I$. Let us consider the function $g \in \mathscr{F}  $ defined by $g (n) = f(n+k)$. The function $g$ is natural. So there exists $d \in \I$ such that $g(d)$ is not a prime number. So $q:= d+k$ is such that $q > k$ and $f(q)$ is not a prime number. This finishes the proof.

\end{proof}

\begin{coro} \label{Fermatcor}Assume Conjecture \ref{conj} is true. Then there are infinitely many composite Fermat's numbers.
\end{coro}

\begin{proof}Apply \ref{infini} to Fermat's natural function $n \mapsto 2^{2^n}+1$.

\end{proof}

\begin{prop} \label{incredi} Assume Conjecture \ref{conj} is wrong. Then there exists an "arithmetical" (relying only on $+,\times ,  \wedge$) formula giving arbitrary big prime numbers.

\end{prop}

\begin{proof}If Conjecture \ref{conj} is wrong. There exists a non-constant (and thus strictly increasing) natural function giving only prime numbers.  
\end{proof}

\begin{rema} In order to define natural functions in Definition \ref{natfct}, we have used $+,\times , \wedge$. One can define with the same formalism a larger class of function using moreover Knuth's up arrow $\{ \uparrow ^n \mid n \geq 1\}$ notation for hyperoperation $(\uparrow ^1 = \wedge )$. Then one can extend Conjecture \ref{conj} to this larger class of functions (i.e. the set of prime numbers is "super-Knuth"). One may also consider to use the factorial symbol. 
\end{rema}

\section{Experimental Results}

\label{secexp}
We checked Conjecture \ref{conj} on many non-polynomial functions using CoCalc. 
When we take a random natural function $(e.g.~ 3^{3^n}+1)$, obviously, the first value ($n=1$) is already very often composite.
In the following, we experiment some (non-polynomial) functions whose first outputs are prime numbers. For each function $f$, we compute the smallest $n\geq 1$ such that $f(n)$ is not prime. We also provide the factorization of $f(n)$.
\[
 \includegraphics[scale=0.50]{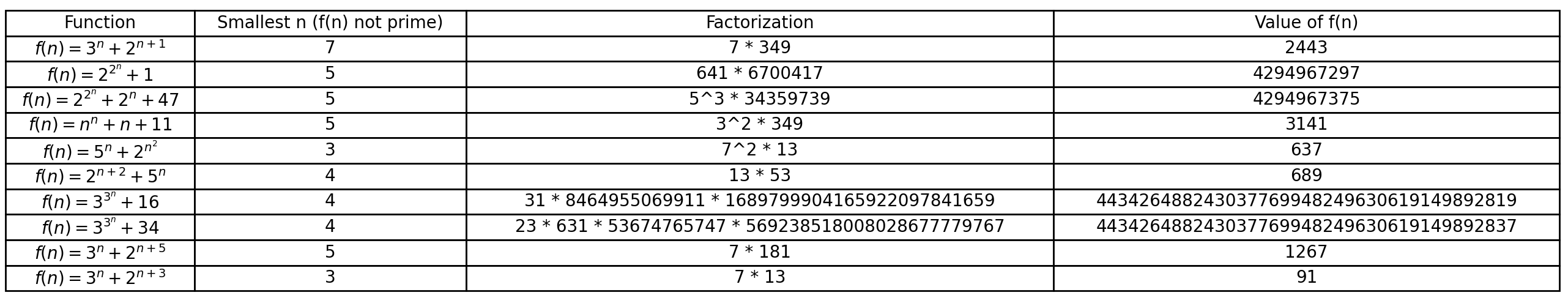}
\]

To us, Fermat numbers do not have something very special. Some other similar natural functions seem to behave similarly. The function $f_{46} : n  \mapsto 2^{2^n}+ 93$ is better than Fermat's function at giving prime numbers (it gives prime numbers for $1 \leq n <7$): 

\[
 \includegraphics[scale=0.6]{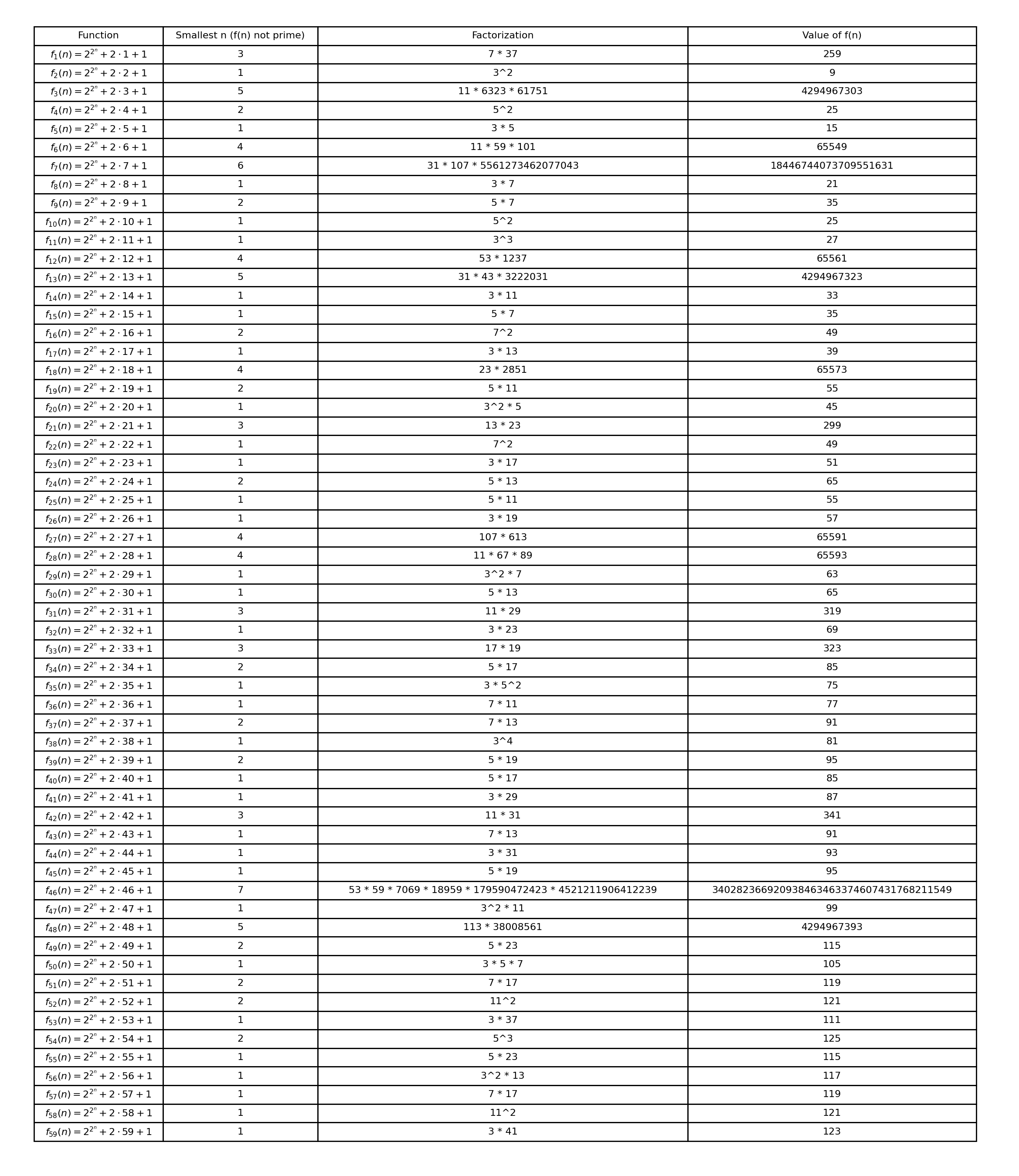}
\]
We tested other functions $n\mapsto 2^{2^n}+2k+1$ for $0\leq k \leq 1300$. Having $f(1), \ldots , f(4)$ primes is relatively frequent. We found another one such that $f(n) $ is prime for all $1\leq n <7$: 
\[
 \includegraphics[scale=0.51]{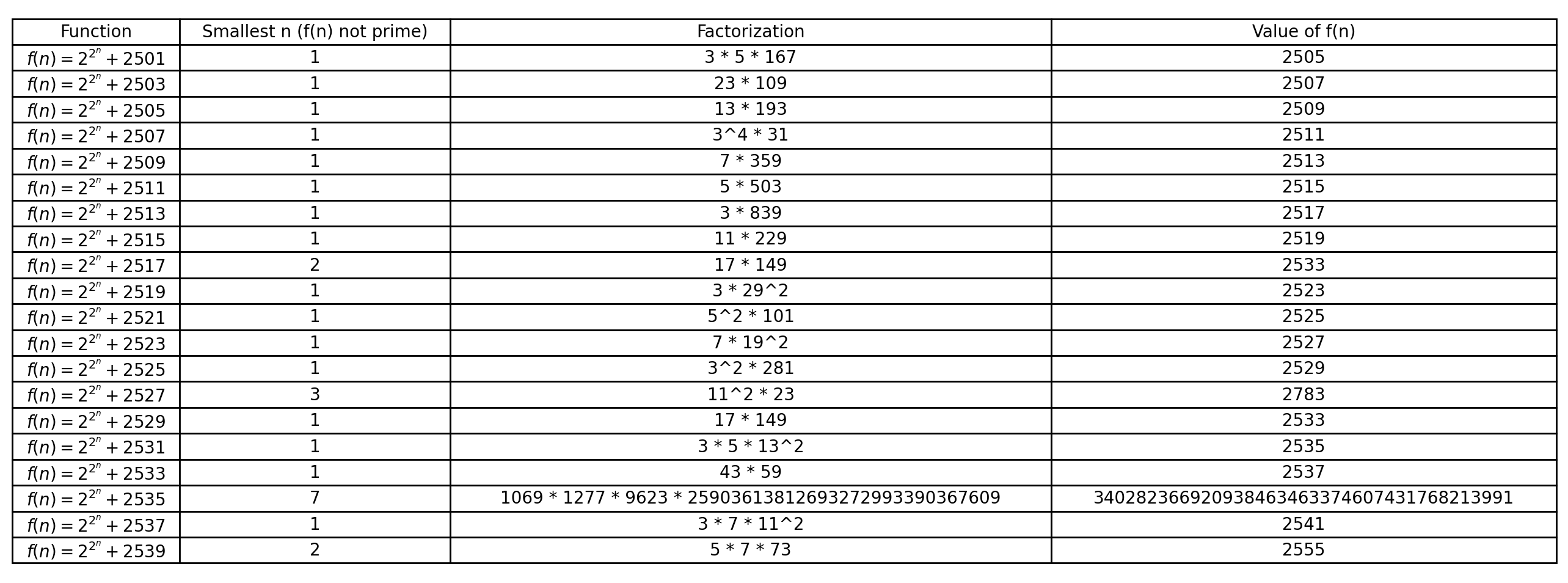}
\]

All these computations provide evidence supporting the conjecture. We conclude this note with the following questions:

\begin{enumerate}
    \item For each \( m > 0 \), does there exist a non-polynomial (resp. double exponential) function  whose first \( m \) outputs are prime numbers?
    \item Is there a canonical form for natural functions, analogous to the unique representation of polynomials as sums of monomials?
\end{enumerate}

We plan to continue these experiments and study these questions elsewhere.

\end{document}